\documentclass[letterpaper,11pt]{article}
\usepackage{amsmath, amsthm, amssymb}    
\usepackage{bridges}			
\usepackage{graphicx}			
\usepackage[colorlinks=true, urlcolor=blue, citecolor=black, linkcolor=black]{hyperref}  
\usepackage{subcaption}			
\usepackage{multirow}
\newtheorem{theorem}{Theorem}

\title{Interleaved Friezes: Celtic Knotwork and Hitomezashi}

\author{Katherine A. Seaton
\vspace{10pt}\\
La Trobe University, Australia; k.seaton@latrobe.edu.au} 

\date{}					

\begin{document}

\maketitle

\thispagestyle{empty}

\begin{abstract}

Both Celtic knotwork and strips of hitomezashi stitching can be interpreted as being two-sided friezes wherein the patterns on the sides are interleaved. We prove which of the thirty-one two-sided friezes can, and cannot, be realized in hitomezashi, and compare this to the case of Celtic knotwork.

\end{abstract}

\section*{Two-Sided Friezes}

A \emph{frieze} is an infinitely long strip decorated with a pattern that repeats periodically by \emph{translation} along the axis of the strip. Physical friezes such as we see decorating contemporary or ancient architecture, homewares, or textiles of many cultures \cite{Symm} have finite length. We politely disregard this when we analyze a frieze mathematically, and treat the pattern as if it repeats indefinitely in one dimension. The width of a physical or mathematical frieze is finite, and generally we think of it as having no thickness. 

The pattern on a frieze may appear unchanged i.e., be symmetric, under one or more isometric transformations: \emph{rotation} through 180$^\circ$, \emph{reflection} across or perpendicular to the strip's long axis (which we take to be horizontal), and \emph{glide reflection}. A {glide reflection} is a single transformation, the net result of reflection across the frieze axis accompanied by a translation parallel to it (the glide). The two friezes shown in Figure~1, adapted from a hitomezashi wallpaper pattern devised by Sen and Martinez \cite{Sen}, have glide reflection as the only symmetry. 
Translation and rotation can actually be performed upon a physical object, but the effect of a reflection can be seen only with the mind's eye. There are exactly seven collections of the isometric transformations which are consistent with the group axioms, and these are the well-known \textit{frieze groups}. 

\begin{figure}[h!tbp]
	\centering
	\includegraphics[height=1.4in]
    {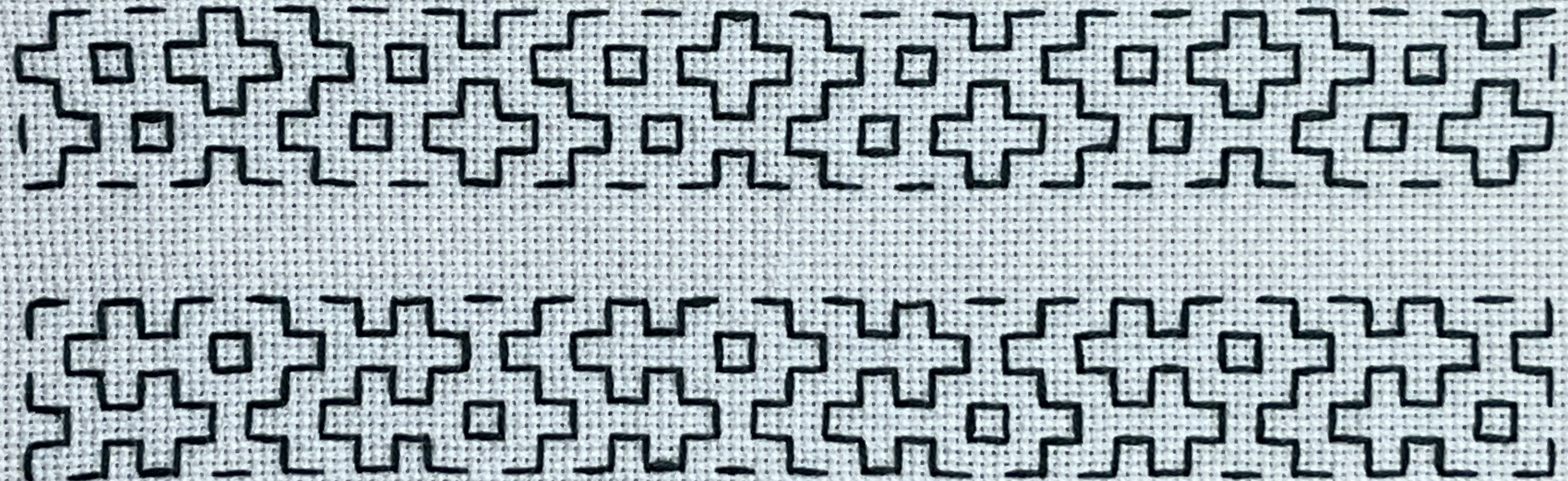}
	\caption{Two hitomezashi friezes with glide reflection symmetry.}
	\label{p1a1}
\end{figure}

Readers may also be aware of the two colour or \textit{colour-reversing frieze groups}, of which there are seventeen \cite{Symm}. Less well-known are the symmetry groups of \textit{two-sided friezes}. Farris and Rossing \cite{FarrisRossing} constructed and considered woven rope friezes representing each of the one-colour one-sided friezes, while noting that such analysis did not fully capture the three-dimensional features of these interwoven friezes. As Cromwell pointed out, even in illuminated manuscripts decorated with borders of Celtic knotwork, parts of the pattern seem to our senses to live just above the two-dimensional page, and other parts appear to live just below it \cite{Cromwell}. For a frieze which has two sides there are more possible isometries to consider, transformations that take the `front' to the `back'. In visualizing these, it is probably helpful to imagine that there is a tiny gap between the two sides of such a frieze. The plane of the frieze is half-way between the two sides, in this gap. 

As well as the familiar 180$^\circ$ rotation about an axis perpendicular to the plane of the frieze, for a two-sided frieze it is possible to perform a 180$^\circ$ rotation either about the horizontal axis or a vertical line in the plane of the frieze. Another transformation which can be physically performed is a \textit{screw} or roto-translation: 180$^\circ$ rotation about the horizontal axis is accompanied by a translation parallel to it. We can also imagine reflection across the plane of the frieze.  Finally, frieze geometry admits a \textit{roto-reflection}, wherein 180$^\circ$ rotation about an axis perpendicular to the plane of the frieze is accompanied by reflection in the plane of the frieze.

There are thirty-one two-sided frieze groups. We defer listing them, because they will appear below in Tables 1 and 2.  In this paper, we use the same notation for them as did Cromwell \cite{Cromwell}, apart from using the more standard lower-case p (he uses P) as the first letter in each label. Each label has the form p\,$\ast\ast\ast$. We specify three axes: $\mathbf{a}$ lies in the plane of the frieze and is the horizontal axis of the frieze; $\mathbf{b}$ lies in the plane of the frieze and is perpendicular to $\mathbf{a}$; $\mathbf{c}$ is perpendicular to the plane in which the frieze lies. In the four symbol labelling, the second position corresponds to behaviour relative to axis $\mathbf{a}$, the third to $\mathbf{b}$ and the fourth to $\mathbf{c}$. The following symbols are used to describe the symmetry behaviour relative to the axes:
\begin{itemize}
\item 1 is a place-holder, indicating no particular symmetry behaviour relative to the axis;
    \item `m' indicates  reflection symmetry in a plane \textit{perpendicular to} the axis (N.\ B. not across this axis);
    \item `a' indicates  glide reflection symmetry, the reflection being in a plane perpendicular to the axis and the glide being parallel to $\mathbf{a}$;
    \item 2 indicates  rotation symmetry of order 2 (that is, of $180^\circ$ degrees) about the axis;
    \item 2$^\prime$ indicates symmetry under the screw (and can only appear in the second position);
    \item $\tilde{2}$ indicates symmetry under roto-reflection (and can only appear in position four).
    \end{itemize}
Where two of these symbols apply to an axis the symbols are both shown, stacked, in the corresponding position.

\section*{Celtic Knotwork Friezes}
Seven of the two-sided friezes have a straightforward relationship to the one-sided friezes: the two sides are mirror images of each other across the frieze plane, and the labelling has the form p\,$\ast \ast$\,m.  In knotwork, be it physical \cite{FarrisRossing} or drawn \cite{Cromwell}, strands cross over and under each other at an angle. Because of such crossing points, a knotwork frieze cannot exhibit reflection symmetry across the frieze plane. 


One way to draw a Celtic knotwork frieze is to draft a panel of plaited strands using a grid. Some of the crossings are then removed and replaced with bends within two strands, in a regular manner. The knotwork retains a strict alternating structure when strands cross. Cromwell used this feature to argue that fourteen other of the remaining twenty-four two-sided frieze symmetries cannot be realised in Celtic knotwork \cite{Cromwell}. He showed one or more examples for each of the remaining ten,  either adapted from historical artefacts such as The Book of Kells or Celtic stonework, or of his own devising. His findings are summarized within Table~1.

For Celtic knotwork friezes the patterns on the two sides are interleaved and are not independent of one another.  A similar relationship exists between the two sides of a piece of hitomezashi stitching. The main purpose of this paper is to apply analysis similar to that of Cromwell for Celtic knotwork and that of Sen and Martinez \cite{Sen} for one-sided hitomezashi wallpaper patterns to two-sided hitomezashi friezes. (A full consideration of hitomezashi wallpaper patterns as two-sided would, as the discussion of chain mail by Farris did \cite{FF2},  invoke the eighty layer groups,  well beyond the scope of eight pages!) We find that the constraints arising for hitomezashi are not the same as those for Celtic knotwork, so that some different frieze patterns can be realized. Table 1 is a summary of the friezes which can be stitched, the friezes displayed in \cite{Cromwell}, and those for which neither this work or Cromwell's can supply an example.

\begin{table}[h!tbp]
\centering
\caption{Realizations of Two-Sided Frieze Symmetries in Celtic Knotwork and in Hitomezashi.}
\begin{tabular}{|c|c|c||cc|}
\hline
\rule{0pt}{2.5ex}Hitomezashi only& Celtic Knotwork only& \hspace{.8cm}Both\hspace{.8cm} &\multicolumn{2}{|c|}{Neither}
\\ 
\hline
p11a&p121&p111&p2aa&p11m\rule{0pt}{3ex}\\
p11$\stackrel{2}{\textrm{a}}$&p1$\stackrel{2}{\textrm{a}}$1&p112&p$\stackrel{2}{\textrm{m}}11$&p2mm\\
p1m1& p211&p1a1&p1$\stackrel{2}{\textrm{m}}$1&p11$\stackrel{2}{\textrm{m}}$\\[1.2ex]  
 pm11&p222&p11$\tilde{2}$&pm2a&p2$^\prime$am\\
 p$\stackrel{2^\prime}{\textrm{m}}$11&p2$^\prime$22&p2$^\prime$11 &pmma&pm2m\\[1.2ex]
p2$^\prime$ma &&  & pmaa &pmmm\\[1.2ex] 
 pmm2 & &  & &pmam\\[1.2ex]
 pma2&&&&\\[1.2ex]
\hline
\end{tabular}
\end{table}

\section*{Hitomezashi Friezes}

Hitomezashi is a form of traditional Japanese stitching, born from frugality in the Edo period but now used for decoration as well as visible mending. In recent years, its intrinsic mathematical properties have been studied and the number of related papers that have appeared since \cite{HayesSeat} has continued to grow year-on-year. Of direct relevance to the analysis we are about to embark on are \cite{SeatHayes}, \cite{Seaton} and \cite{Sen}.

We have already seen friezes stitched in hitomezashi in Figure 1. We notice that they are comprised of horizontal and vertical lines of running stitch on a square grid. Stitches and gaps alternate strictly in each line of stitching. At each vertex of the grid, apart from along the top and bottom of the friezes, one horizontal stitch and one vertical stitch form a right-angle corner. 

 On the other side of the fabric, a gap lies behind each stitch on the front, and a stitch lies behind each gap on the front. When we stitch any hitomezashi pattern on one side of a piece of fabric we get a second one on the other side. (Not all forms of embroidery have this property.) This second pattern is the complement or dual of the first, and it may look quite different. We get no choice about what it looks like, because it is completely determined by the stitching we placed on the front. In fact, the friezes in Figure 1 are dual to each other! 

 It is possible for a hitomezashi pattern to be self-dual under translation \cite{SeatHayes} or rotation \cite[p. 119]{Seaton}. That is, the pattern on the reverse of the fabric may be a shifted or rotated copy of the pattern on the front. For this reason, considering a pattern together with its dual as a two-sided frieze should give examples of at least some of the thirty-one two-sided friezes.  The frieze in Figure 1 is an example of p1a1 (no particular symmetry apart from the glide reflection in the plane perpendicular to axis $\mathbf{b}$).

 Note that in order to display both sides of the two-sided frieze in the correct relative position, I have stitched two versions of it one below the other, showing the obverse and reverse both facing the viewer. The `reverse' is not what one would see by picking up the upper length of hitomezashi from Figure 1 and turning it over, because that introduces a rotation. Rather, the lower stitching is what one would see in a mirror placed behind the upper frieze if one peeped around the fabric. The top, bottom, left and right of the obverse and reverse are in correct orientation to each other. For a fuller explanation of duality see \cite{SeatHayes} or \cite[Ch. 4]{Seaton}.

The complementary nature of the stitching on the two sides prevents hitomezashi two-sided friezes from being symmetric under some groups of transformations. In order to prove which these are, mathematical notation to specify hitomezashi stitch patterns introduced in \cite{HayesSeat} is helpful.

\newpage
Because the first stitch in each line of stitching determines all the others, we need only keep track of whether the first stitch is present (1) on the front or not present on the front (0). If the first stitch is present (respectively not present), then so too is the third, fifth etc., while the second, fourth etc. are not present (respectively present). One binary word $x$ is used to specify the state of the vertical lines of stitching; for friezes this word will be periodic. A second binary word $y$ of finite length can be used to specify the state of the horizontal lines of stitching in the frieze. For the reverse of such a binary word $z$ we write $z^R$, for the binary complement ($0 \leftrightarrow 1$ in $z$) we write $z^C$ and for the length of $z$ we write $|z|$.

In their study of hitomezashi wallpaper patterns, Sen and Martinez \cite{Sen} established results equivalent to the following:
\begin{enumerate}
    \item For any binary word $z$, $z\neq z^C$. Further, if $|z|$ is odd, $z\neq z^{RC}$.
    \item There is reflection symmetry in the plane perpendicular to $\mathbf{a}$ iff $x=x^R$ and  $|x|$ is even.
\end{enumerate}

\begin{theorem} None of the two-sided frieze patterns with the symbol \textrm{`m'} in the fourth position can be realised in hitomezashi. 
\end{theorem}
\begin{proof} We compare the pattern specified by $x$ and $y$ to its complement on the other side of the frieze which is specified by $x^C$ and $y^C$. These are not the same pattern, by (I).
\end{proof}

\begin{theorem} Frieze patterns with labelling of the form p2$\ast\ast$ or p{$\ast$}2$\ast$ cannot be realised in hitomezashi.

\end{theorem}
\begin{proof} The total number of stitches and gaps in the vertical lines of stitching is $|y|-1$.  Under 180$^\circ$ rotation about $\mathbf{a}$ 
\[
x \mapsto \begin{cases}x&|y|\textrm{ even} \\
x^C&|y|\textrm{ odd} \end{cases} \qquad \textrm{and}\qquad y \mapsto y^R.
\]
We compare the pattern thus generated with that specified by $x^C$ and $y^C$. These patterns cannot be same if $|y|$ is even, by (I). If $|y|$
is odd, we compare $y^R$ and $y^C$, but they are not equal by the second part of (I).

An analogous argument applies for rotation about $\mathbf{b}$.
\end{proof}

\begin{theorem} Frieze patterns with labelling of the form pm$\ast$a cannot be realised in hitomezashi.
\end{theorem}
\begin{proof} Assume such frieze patterns can be realised.
 Because of the reflection indicated by `m' in the second position, the vertical lines of stitching are specified by repetitions of a word with the structure $ww^R$ by (II). Potential glides have length either $|w|$ or $|x|=2|w|$ (because the transformation `a' is an involution). The word  $y$ is unchanged under an even-length translation parallel to $\mathbf{a}$, and is mapped to $y^C$ under an odd-length translation parallel to $\mathbf{a}$.
 
 Apply the transformation indicated by `a' in the fourth position, with the glide equal to $2|w|$: 
 \[
ww^R \mapsto ww^{R}\qquad \textrm{and}\qquad y \mapsto 
y.
\]
We compare the pattern thus generated with that specified by $w^Cw^{RC}$ and $y^C$. But $y \neq y^C$ by (I). If any such patterns exist, the glide must be by $|w|$. Applying the same transformation with this glide length, 
 \[
ww^R \mapsto w^{R}w\qquad \textrm{and}\qquad y \mapsto \begin{cases}
y&|w| \textrm{ even}\\y^C& |w|\textrm{ odd}
\end{cases}.
\]
We compare the pattern thus generated with that specified by $w^Cw^{RC}$ and $y^C$. Because $y \neq y^C$ by (I), if such patterns exist $|w|$ is not even.  But for $|w|$ odd, $w \neq w^{RC} $ by the second part of (I). 

Thus such frieze patterns cannot be created in hitomezashi.
\end{proof}


These results, summarised in Table 2, rule out hitomezashi realisations of eighteen of the thirty-one two-sided friezes. Of course, to establish that the remaining friezes can be realized, an existence proof is required. This is to be found in Figures~1--13, stitched by the author on Aida band with embroidery floss. Within the stitching we see some traditional hitomezashi motifs \cite{SeatHayes}: small squares (mouths), crosses, castellations, linked steps, and  mountain form. The overbar notation used in the captions indicates the repeating unit in the periodic word $x$.
Figure~2 shows a frieze of the type p1m1, which can instructively be compared to p1a1 in Figure~1.
Figures~3--7 show the five friezes with labels of the form p11$\ast$, again inviting comparison; they provide examples of various symmetries relative to $\mathbf{c}$. The only two-sided frieze with the roto-reflection $\tilde{2}$ is in this group.
In Figures 8--11 we display friezes with symmetry under vertical reflection.
Finally, Figures~11--13 provide examples of friezes symmetric under the screw transformation.

\begin{table}[h!tbp]
\centering
\caption{Two-Sided Frieze Symmetries Incompatible or Compatible with Hitomezashi Constraints.}
\begin{tabular}{|c|cc|c||cc|}
\hline
\multicolumn{4}{|c||}{Incompatible}\rule{0pt}{2.5ex}&\multicolumn{2}{|c|}{\multirow{2}{*}{Compatible}}\\
\cline{1-4}
  
Theorem 1 \rule{0pt}{2.5ex}& \multicolumn{2}{|c|}{Theorem 2} & Theorem 3 & &  \\
\hline
\multicolumn{2}{|c}{\phantom{XXX}{p2mm}}&&&p1a1&pma2\rule{0pt}{2.5ex}\\[1.2ex]
\multicolumn{2}{|c}{{\phantom{XXX}}{pm2m} }&&&p1m1  &pmm2\\[1.2ex]
pmmm &{p2aa}&p211&&  p11a &pm11\strut\\[1.2ex]
pmam &p222&p$\stackrel{2}{\textrm{m}}11$&& p11$\stackrel{2}{\textrm{a}}$&p$\stackrel{2^\prime}{\textrm{m}}$11\\[1.2ex]
p11m &p121 &p2$^\prime$22& pmma& p112 &p2$^\prime$ma \\[1.2ex]

p11$\stackrel{2}{\textrm{m}}$  & p1$\stackrel{2}{\textrm{m}}$1&p1$\stackrel{2}{\textrm{a}}$1&pmaa&p11$\tilde{2}$&p2$^\prime$11\\
 p2$^\prime$am&&\multicolumn{2}{c||}{pm2a\phantom{XXX} }\rule{0pt}{3ex}&p111&\\[1.2ex]
\hline
\end{tabular}
\end{table}

\begin{figure}[h!tbp]
	\centering
	\includegraphics[width=0.95\textwidth, angle=180]{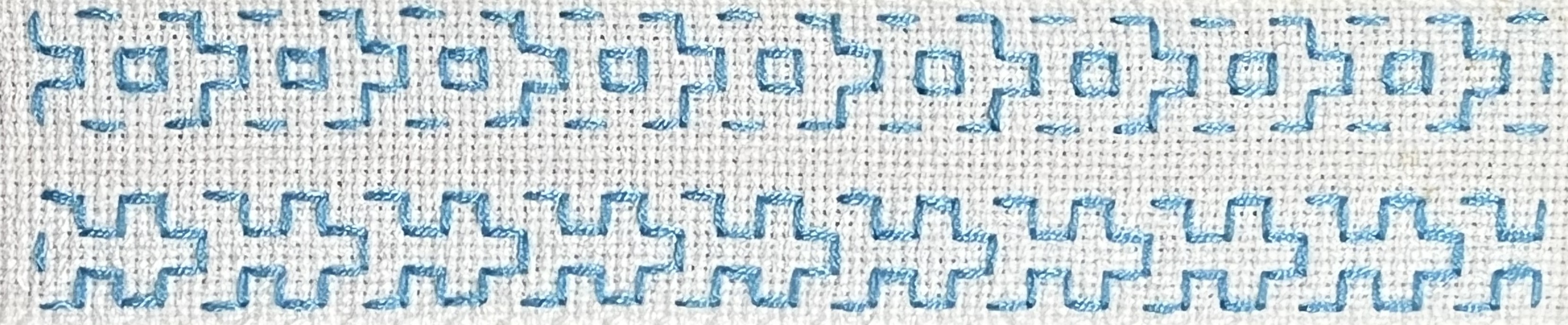}
	\caption{The two sides of a p1m1 hitomezashi frieze; $x=\overline{1110}$, $y=0110$.\\ Compare with p1a1 in Figure 1 where $x=\overline{1000110}$, $y=100110$. }
	\label{p1m1}
\end{figure}

\begin{figure}[h!tbp]
	\centering
	\includegraphics[width=0.95\textwidth, angle=180]{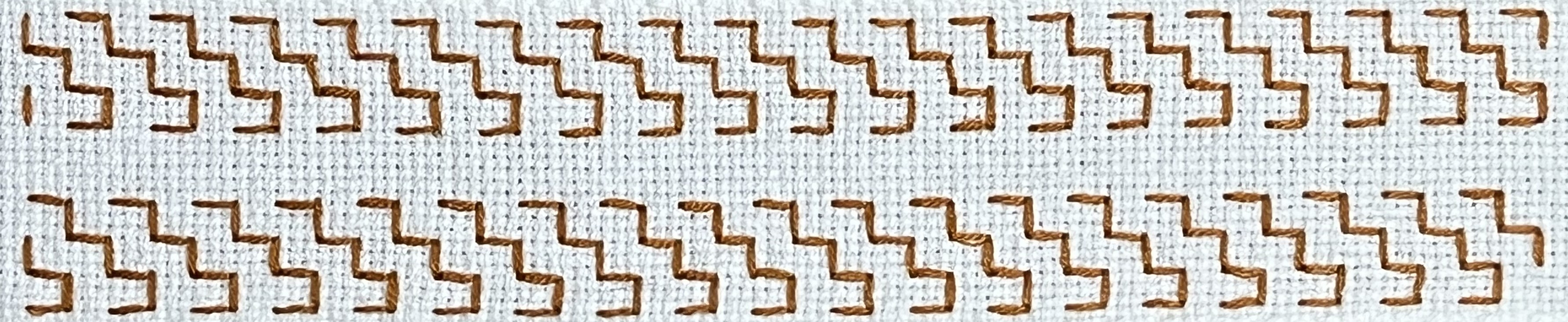}
	\caption{The two sides of a p11a hitomezashi frieze. $x=\overline{01} $, $y=0100$.}
	\label{p11a}
\end{figure}

\begin{figure}[h!tbp]
	\centering
	\includegraphics[width=0.95\textwidth, angle=180]{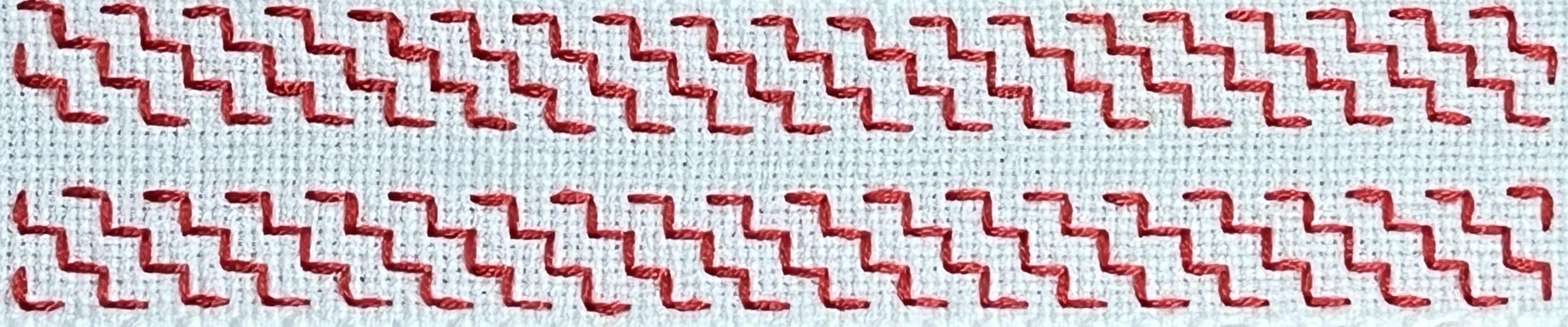}
	\caption{The two sides of a p11$\stackrel{2}{\textrm{a}}$ hitomezashi frieze. $x=\overline{10} $, $y=1010$.\\ This is  the traditional dan tsunagi (linked steps) stitch.}
	\label{p112a}
\end{figure}

\begin{figure}[h!tbp]
	\centering
	\includegraphics[width=0.95\textwidth]{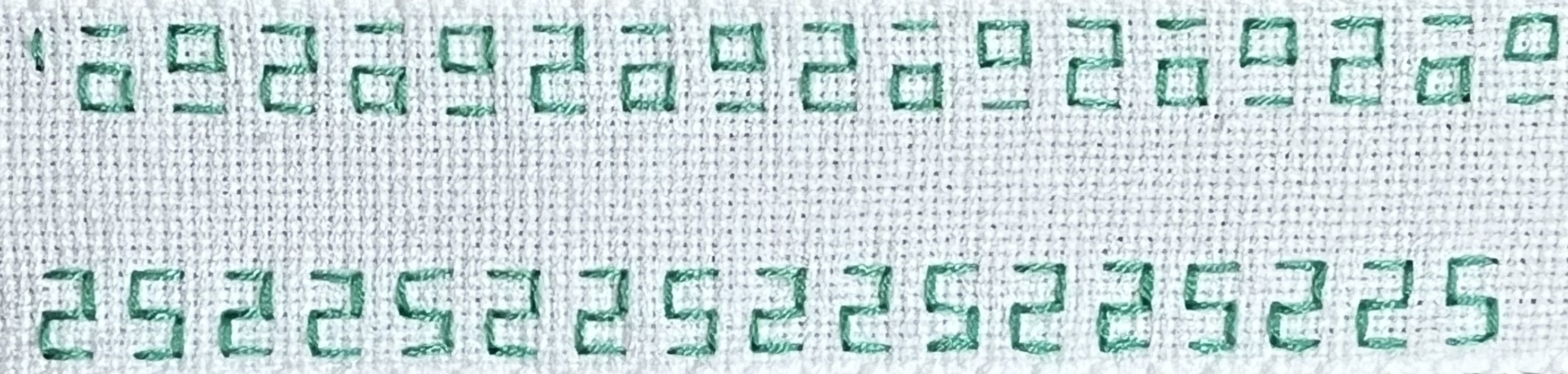}
	\caption{The two sides of a p112 hitomezashi frieze. $x=\overline{011001} $, $y=000$. }
	\label{p112}
\end{figure}

\begin{figure}[h!tbp]
	\centering
	\includegraphics[width=0.95\textwidth]{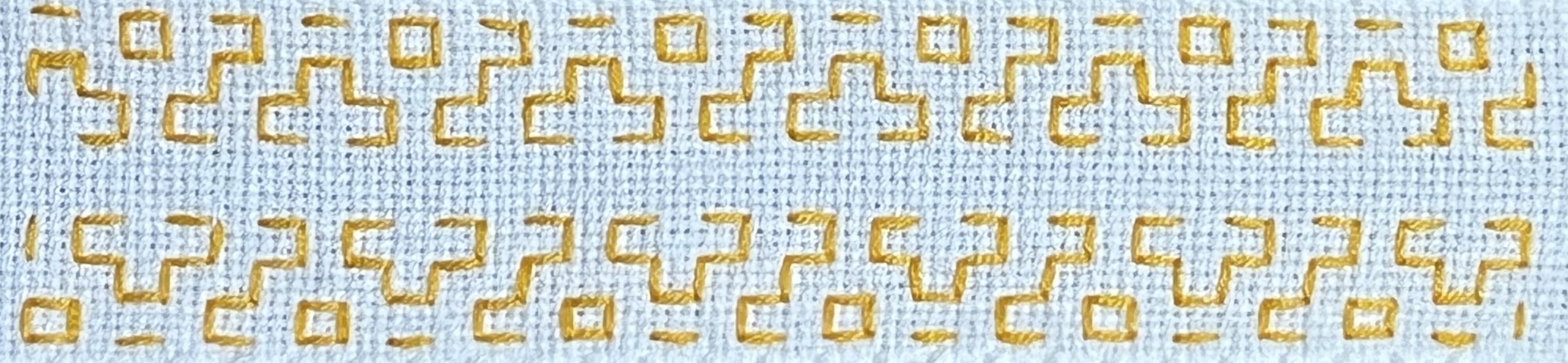}
	\caption{The two sides of a p11$\tilde{2}$ hitomezashi frieze. $x=\overline{001101} $, $y=0011$.}
	\label{p112tilde}
\end{figure}

\begin{figure}[h!tbp]
	\centering
	\includegraphics[width=0.95\textwidth]{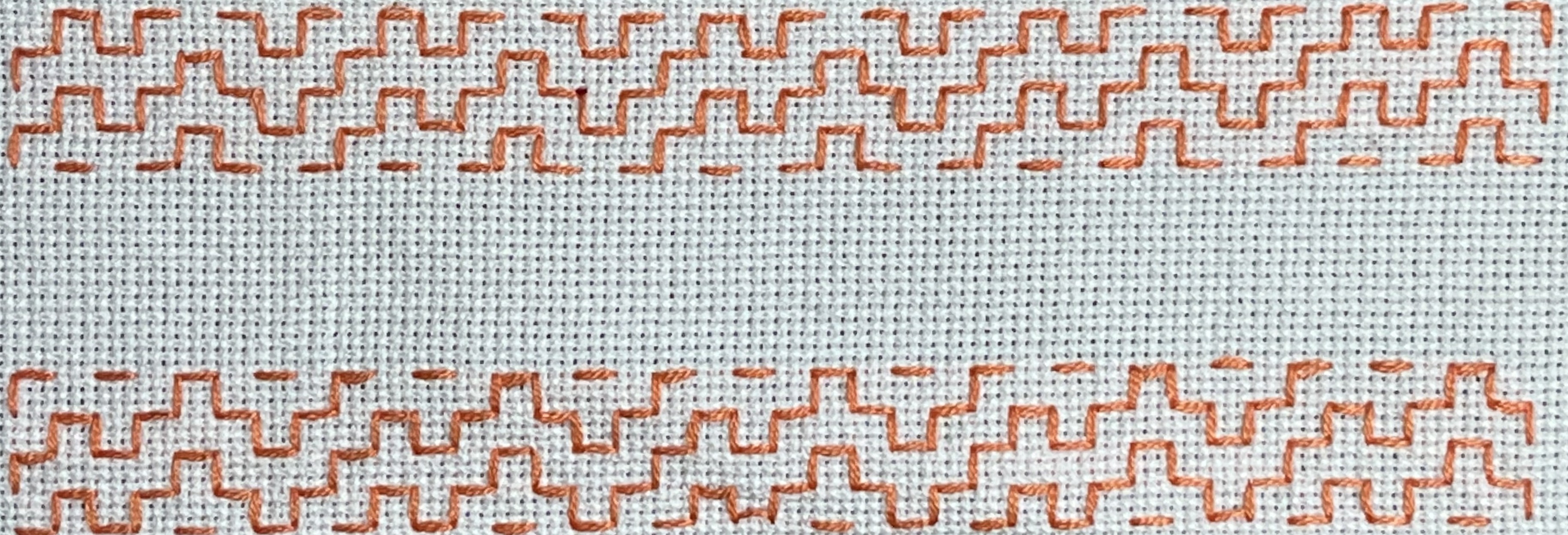}
	\caption{The two sides of a p111 hitomezashi frieze. $x= \overline{10001100}$, $y=01010$.}
	\label{p111}
\end{figure}

\begin{figure}[h!tbp]
	\centering
	\includegraphics[width=0.95\textwidth, angle=180]{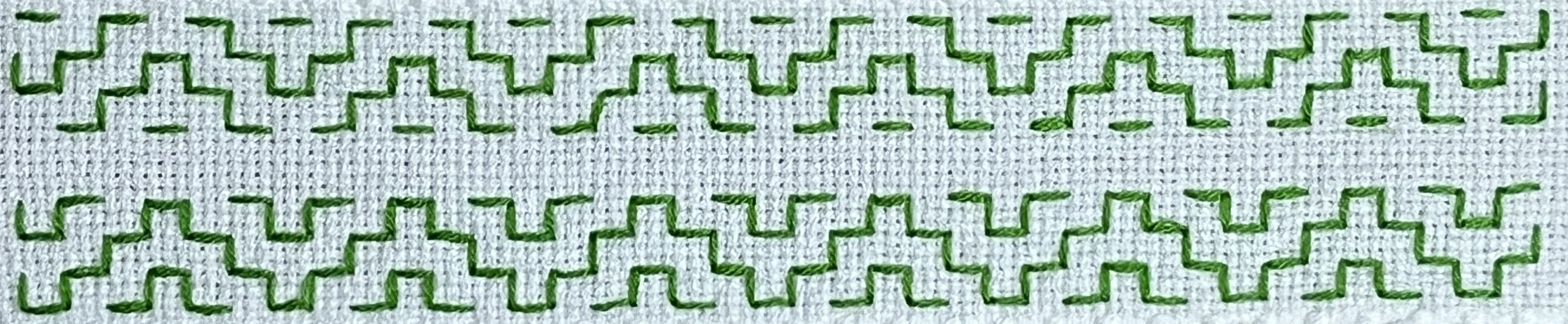}
	\caption{The two sides of a pma2 hitomezashi frieze. $x= \overline{001}$, $y=1010$.}
	\label{pma2}
\end{figure}

\begin{figure}[h!tbp]
	\centering
	\includegraphics[width=0.95\textwidth]{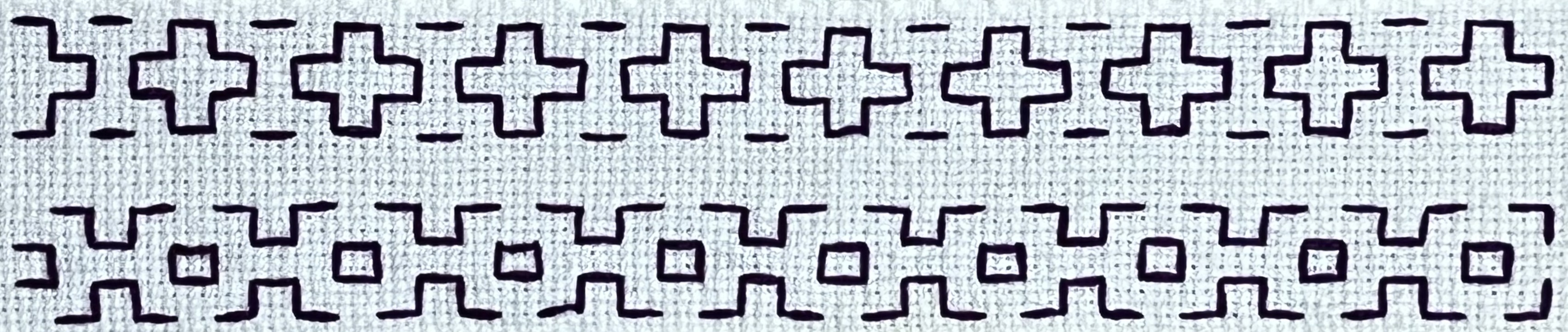}
	\caption{The two sides of a pmm2 hitomezashi frieze. $x=\overline{1001}$, $y=1001$. }
	\label{pmm2}
\end{figure}

\begin{figure}[h!tbp]
	\centering
	\includegraphics[width=0.95\textwidth]{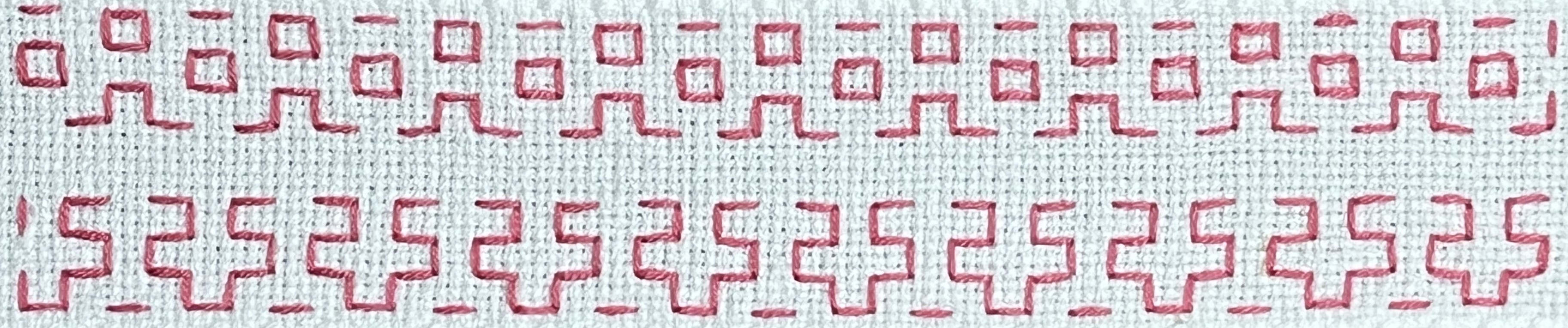}
	\caption{The two sides of a pm11 hitomezashi frieze. $x= \overline{0011}$, $y=0111$.}
	\label{pm11}
\end{figure}

\begin{figure}[h!tbp]
	\centering
	\includegraphics[width=0.95\textwidth]{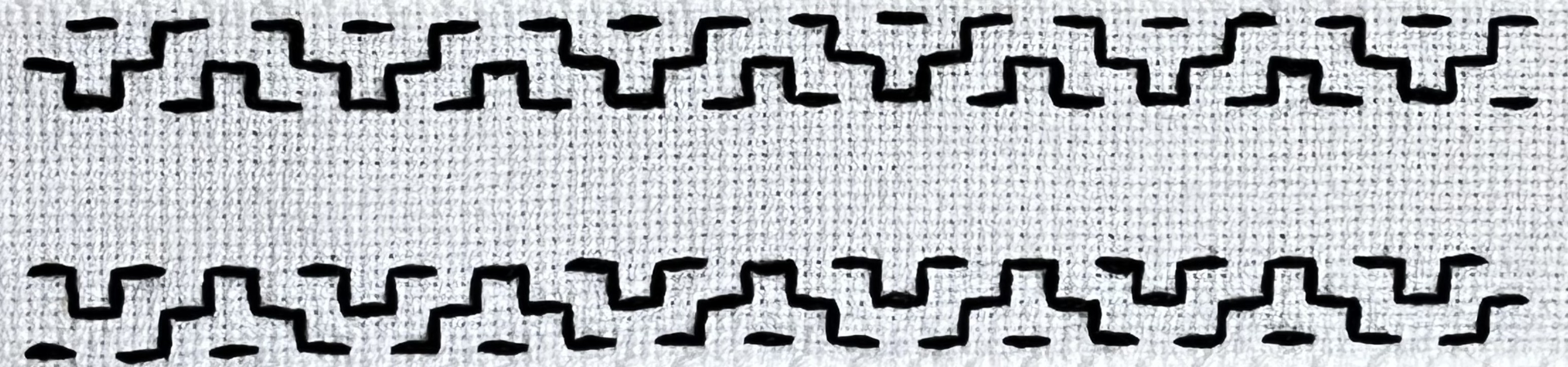}
	\caption{The two sides of a p$\stackrel{2^\prime}{\textrm{m}}$11 hitomezashi frieze. $x=\overline{011} $, $y=010$.}
	\label{p2primem11}
\end{figure}
\newpage
\begin{figure}[h!tbp]
	\centering
	\includegraphics[width=0.95\textwidth]{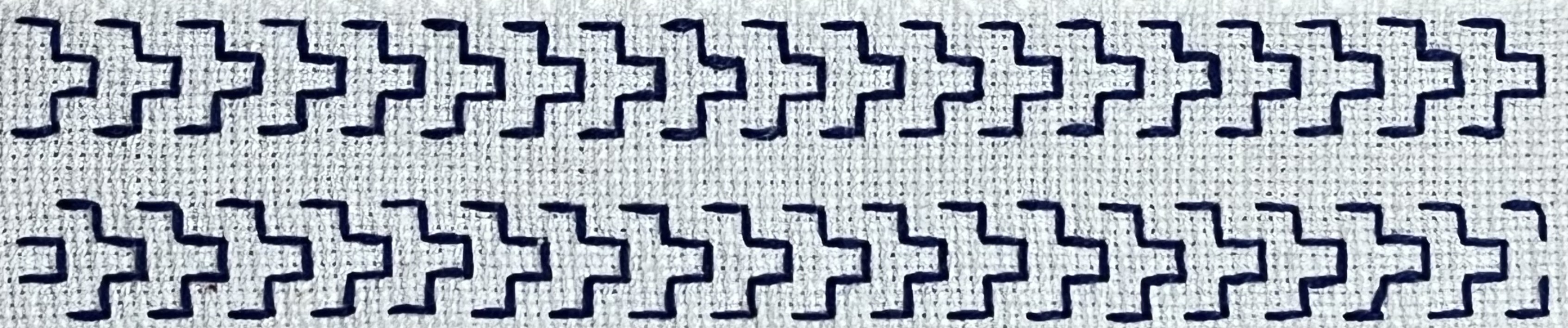}
	\caption{The two sides of a p2$^\prime$ma hitomezashi frieze. $x=\overline{01} $, $y=1001$.}
	\label{p2primema}
\end{figure}

\begin{figure}[h!tbp]
	\centering
	\includegraphics[width=0.95\textwidth, angle=180]{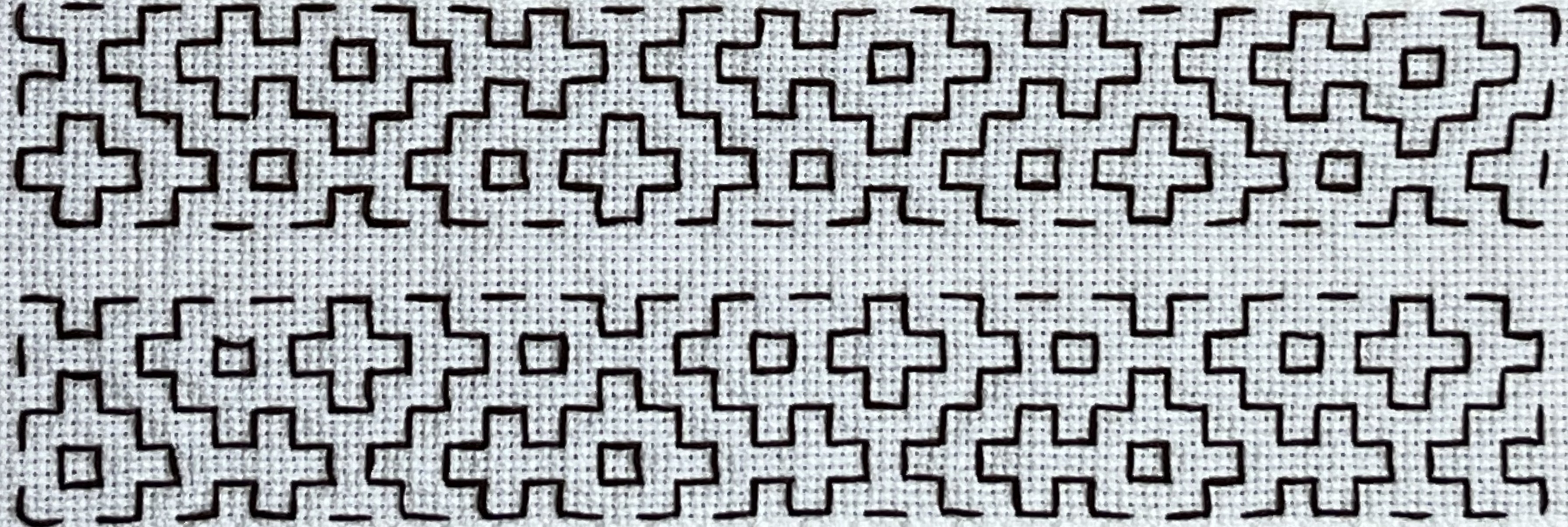}
	\caption{The two sides of a p2$^\prime$11 hitomezashi frieze. $x=\overline{0101100} $, $y=0110110$.}
	\label{p2prime11}
\end{figure}


    
{\setlength{\baselineskip}{13pt} 
\raggedright				

} 
   
\end{document}